\documentclass[12pt,a4paper]{amsart}
\usepackage{amssymb,enumerate,mathtools,cite}
\usepackage{graphicx,epstopdf,wasysym}
\usepackage{caption,subcaption,float}
\usepackage{hyperref} 
\numberwithin{equation}{section}
\newtheorem{theorem}{Theorem}[section]
\newtheorem{lemma}[theorem]{Lemma}

\theoremstyle{remark}
\newtheorem{remark}{Remark}[section]
\newtheorem{definition}{Definition}[section]
\newcommand{\D}{\mathbb{D}}

\allowdisplaybreaks

\DeclareMathOperator{\RE}{Re}

 \def \a{\alpha}

 \title[Generalized radius problems for quotient functions]{Generalized radius problems for quotient functions having fixed second coefficient}

 \author[S. Anand]{Swati Anand}

 \address{Department of Mathematics, University of Delhi,
 Delhi--110 007, India}
 \email{swati\_anand01@yahoo.com}
 
 \author [N.K. Jain]{Naveen Kumar Jain}
 \address {Department of Mathematics, Aryabhatta College, Delhi-110021,India}
 \email{naveenjain05@gmail.com}

  \author [S. Kumar]{Sushil Kumar}
 \address {Bharati Vidyapeeth's college of Engineering, Delhi-110063, India}
 \email{sushilkumar16n@gmail.com}

\begin{document}


\maketitle

\begin{abstract}
In this manuscript, we deal with three classes of quotient functions having fixed second coefficient described on open unit disk. The radius of strongly starlikeness,  lemniscate starlikeness, lune starlikeness,  parabolic starlikeness, sine starlikeness, exponential starlikeness and several other radius estimates for  such classes are examined.
Relevant relations of obtained radius estimates with the existing estimates are also discussed.
\end{abstract}

\section{Preliminary Literature}\label{sec1}
Bieberbach theorem gives the bound on second coefficient in the Maclaurin series expansion of the univalent functions which has the significant part in univalent function theory. For that reason,
the functions with fixed second coefficients has attracted the interest of many authors over the decades.
Initially, Gronwall~\cite{Gronwall} investigated a subclass of  functions $f\in\mathcal{S}$ with fixed second coefficient. In 2011, Ali \emph{et al.}~\cite{ANV} discussed second order differential subordination in the context of the analytic functions having  fixed second coefficient.
In~\cite{lee}, Lee \emph{et al.} discussed some useful applications
on differential subordination for such functions. Recently, Kumar \emph{et al.}~\cite{sushil17} computed the sharp bound on
initial  coefficients of Ma-Minda type univalent functions with fixed
second coefficient.
Several authors have worked in the direction of growth and distortion  estimates, subordination, Livingstons problem and radius problems for functions with fixed second coefficient. For related literature, see~\cite{Ural87,kumar83,mendiratta1,pas89}.
The class of the  analytic functions $f$ which are  normalized by the condition $f(0)=f'(0)-1=0$ and defined on  $\D= \{z\in \mathbb{C}: \vert z\vert <1\}$ is denoted by $\mathcal{A}$ and the class of all  functions $f\in \mathcal{A}$ of the form
\[f(z) = z+bz^2+ \cdots \quad\quad (z \in \D)\] is denoted by $\mathcal{A}_b$ where $b$ is the fixed number satisfying the inequality $0 \le b\le1$.
Let $\mathcal{S}$ be the subclass of $\mathcal{A}$ which contains  all univalent functions. Let $\mathcal{P}$ be the class of analytic functions having positive real part. Let $\mathcal{S}^*$ and $\mathcal{K}$ represent the subclasses of $\mathcal{S}$ containing starlike and convex functions, respectively. Analytically $f\in \mathcal{S}^*$ if $zf'(z)/f(z)\in \mathcal{P}$ and $f\in \mathcal{K}$ if $1+zf''(z)/f'(z)\in \mathcal{P}$ for all $z\in \mathbb{D}$. The function $f\in \mathcal{A}$ is acknowledged as close to convex if there exists a function $g\in \mathcal{K}$ such that $f'(z)/g'(z)$ has positive real part for all $z \in \mathbb{D}$. These functions were studied initially by Kaplan~\cite{kaplan}.
Let ${P}$ be a property and $\mathcal{M}$ be a set of functions. Then the real  number
 $R_{P}(\mathcal{M})=\sup \{ r>0:\,  f \in \mathcal{M} \,\mbox{has the property}\,P\, \mbox{in the disk}\,\mathbb{D}_r \}$
is radius of  the property ${P}$ for the set $\mathcal{M}$. For the class $\mathcal{S}$, the sharp radius of convexity $R_{\mathcal{K}}(\mathcal{S})$ is $2-\sqrt{3}$ due to  Nevanlinna and radius of starlikeness $\mathcal{R}_{\mathcal{S}^*}(\mathcal{S})=\tanh\frac{\pi}{4}\approx 0.65579$  due to  Grunsky.
In~\cite{ali1}, Ali \emph{et al.} determined certain radius of starlikeness for the functions $f\in \mathcal{A}_b$. For more literature on radius problems, see~\cite{gupta80, gregor,juneja, SPA21, sokol}.

In 1992,  Ma and Minda~\cite{MM} introduced and studied the covering, growth and distortion theorem for the unified class $\mathcal{S^*(\varphi)}$  of starlike functions which contains the functions $f$ satisfying
${zf'(z)}/{f(z)}\prec\varphi(z)$, $ z\in \mathbb{D}$
where  $\varphi$ is the analytic  function satisfy inequality $\RE(\varphi(z))>0$ and $'\prec'$ shows here subordination. The class $\mathcal{S^*(\varphi)}$  contains various well known subclasses of starlike  functions. For instance, $S^*((1+Az)/(1+Bz)):=S^*[A,B]$;
$(-1 \le B <A\le1)$. The class $S^*[A,B]$ is known as the class of Janowski starlike functions~\cite{janowski73}.
For $0 \le \a <1$, the class $S^*[A=1-2\a, B=-1]:=S^*(\a)$ accommodate starlike functions of order $\a$~\cite{rob}.
For $0 \le \gamma <1$, the class $S^*_{\gamma}:=S^*((1+z)/(1-z)^{\gamma})$ contains strongly stalike functions of order $\gamma$~\cite{Branan69}. These functions satisfy $\vert \mbox{arg} ({zf'(z)}/{f(z)})\vert<({\alpha \pi}/{2})$.
The class $S_p:= S^*(1+({2}/{\pi^2}) (\log (({1-\sqrt{z}})/({1+\sqrt{z}})))^2$ contains parabolic
starlike functions which are associated with  region $\{w\in \mathbb{C}: \vert w-1 \vert < \RE(w)\}$~\cite{Ronn93}. The class $S^*(\sqrt{1+z}):= S^*_L $ contains lemniscate starlike functions~\cite{SS1,sokol}.
The classes $S^*(e^z):=S^*_e$ and $S^*(\sqrt{2}-(\sqrt{2}-1)\sqrt{(1-z)/(1+2(\sqrt{2}-1)z)}):=S^*_{RL}$ which are related with the regions  $\{w\in \mathbb{C}:\vert\log w \vert<1\}$ and $\{w\in \mathbb{C}:\vert(w^2-2\sqrt{2}z+1)\vert<1\}$ respectively~\cite{mnr,mnr1}. The functions $f\in S^*_e$ are known as exponential starlike functions. In~\cite{sharma}, the class $S^*(1+(4/3)z+(2/3)z^2):=S^*_c $ contains the functions $f$ associated with cardioid domain.
In~\cite{raina}, the class $S^*(z+\sqrt{1+z^2}):=S^*_{\leftmoon}$ contains lune starlike functions which are related to region $\{w \in \mathbb{C}:\vert w^2-1\vert <2 \vert w\vert\}$.
The class $S^*(1+\sin z):=S^*_{sin}$ which is  related to the Sin function was studied in \cite{cho}.
The gemetric properties of the class $S^*(1+(z(k+z))/(k(k-z))):=S^*_R$,  where $k = \sqrt{2}+1$ were studied in~\cite{kumar}.
The class $S^*(1+z-(z^3/3)):=S^*_{N_e}$
which is associated with two cusped kidney shaped region $\{u+ ib:((u-1)^2+v^2-4/9)^3-4v^2/3<0\}$~\cite{wani}.
The subordination properties of the class $S^*_{SG}:=S^*(2/(1+e^{-z}))$ which is associated with $\vert\log (w/(2-w))\vert<1$, were discussed in~~\cite{goel}.

\section{Classes of Quotient Functions}\label{sec1}
In~\cite{ali2}, Ali \emph{et al.}  discussed radius problem for the functions $f$ fulfilling the one of the condition (i) $\RE(f(z)/g(z))>0$ (ii) $\vert (f(z)-g(z))/g(z)\vert <1$ under some restrictions on $g$. These functions are closely associated with close-to-convex functions.
Recently, Sebastian \emph{et al.} \cite{AV} computed certain best possible radius estimates for the classes of functions $f \in \mathcal{A}$ fulfilling one of the following conditions:
(i) $\RE f(z)/g(z)>0$ where $(1+(1/z))g(z)$ has positive real part (ii) $\vert (f(z)-g(z))/g(z)\vert>0$ where $(1+(1/z))g(z)$ has positive real part where $g\in \mathcal{A}$, or  (iii) $ (1+(1/z))f(z) $ has positive real part.
Motivated by the aforementioned works, we consider the  three classes of quotient functions having fixed second coefficients which are defined as:

The functions $f$ and $g$  whose Taylor series expansions are given by
\begin{equation}\label{eqf}
f(z) = z+ a_2 z^2+ \cdots
\end{equation}
and
\begin{equation}\label{eqg}
g(z)= z+g_2 z^2+ \cdots
\end{equation}
such that $ (1+z)g(z)/z  $ and  $ f(z)/g(z) $ are in $\mathcal{P}$ for all $z\in \mathbb{D}$. Then \[\dfrac{1+z}{z}g(z) = 1+(1+g_2)z+ \cdots \in \mathcal{P}.\] Thus,  we get $\vert 1+g_2\vert \le2$, that is $-3\leq g_2\ \le 1 $. Also,
$ {f(z)}/{g(z)} = 1+(a_2-g_2)z+ \cdots \in \mathcal{P}$ which gives $\vert a_2-g_2\vert \leq 2$, that is  $-5\leq a_2 \leq 3$. Now, we consider the functions $f$ and $g$ with fixed second coefficient and having following series expansion
\[f(z) = z + 5 b z^2 + \cdots\,\,\mbox{ and}\,\,
g(z) = z+ 3cz^2+ \cdots\]
where $\vert b\vert \le 1$ and $\vert c\vert \le 1$.

\begin{definition}
For $\vert b\vert \le 1$ and $\vert c\vert \le 1$, let $\mathcal{F}_{b,c}^1$ be  defined as
\[\mathcal{F}_{b,c}^1 := \left\{f \in \mathcal{A}_{5b} :\RE \left(\frac{f(z)}{g(z)} \right)>0\, \mbox{and} \RE \left(\frac{1+z}{z}g(z) \right)>0,g \in \mathcal{A}_{3c}, z \in \D\right\}. \]
\end{definition}
Let  $f$ and $g$ be two functions given by (\ref{eqf}) and (\ref{eqg}) such that $\vert(f(z)/g(z))-1\vert <1$ and $\RE ((1+z) g(z)/z) >0$. Since $\left\vert(f(z)/g(z))-1\right\vert <1$ if and only if $\RE(g(z)/f(z)) >1/2$ and so we have $\vert a_2\vert \le 1+ \vert g_2 \vert \le 4$. In the next theorem, we consider the class of functions $f$ and $g$ with fixed second coefficient whose series expansion is given by
$f(z) = z + 4b z^2 + \cdots$ and
$g(z) = z+ 3cz^2+ \cdots$, where $\vert b\vert \le 1$ and $\vert c\vert \le 1$.
\begin{definition}
For $\vert b \vert \le 1$ and $\vert c\vert \le 1$, let the class $\mathcal{F}_{b,c}^2$ be defined as
\[\mathcal{F}_{b,c}^2 := \left\{f \in \mathcal{A}_{4b} :\left\vert\frac{f(z)}{g(z)}-1 \right\vert<1\,\, \mbox{and}\,\, \RE \left(\frac{1+z}{z}g(z) \right)>0,\,\,\,g \in \mathcal{A}_{3c}, z \in \D\right\}.\]
\end{definition}
We take the function $f$  given by (\ref{eqf}) satisfying $\RE ((1+z) f(z)/z) >0$. Then $(1+z) f(z)/z = 1+(1+a_2)z+ \cdots \in \mathcal{P}$ and hence $-3\leq a_2 \leq 1$. Next, we  consider the class of following functions $f$ with fixed second coefficient
\[f(z) = z + 3b z^2 + \cdots \,\,\,\mbox{where}\,\,\, \vert b\vert \le 1.\]
\begin{definition}
For $|b| \le 1$, let $\mathcal{F}_b^3$ be defined as
\[\mathcal{F}_b^3 := \left\{f \in \mathcal{A}_{3b} :\RE \left(\frac{1+z}{z}f(z) \right)>0, z \in \D\right\} .\]
\end{definition}
The function
\begin{equation*} \label{ext3}
f_1(z) = \frac{z(1+(1+3b)z+z^2)}{(1+z)(1-z^2)}.
\end{equation*}
so that
\[\frac{1+z}{z}f_1(z) = \frac{1+w_3(z)}{1-w_3(z)}\]
where \[w_3(z) = \dfrac{z(z+(1+3b)/2)}{(1+(1+3b)z/2)}\,\,\mbox{ with}\,\,\vert 1+3b\vert \le 2.\]
It is observed that $w_3\in \mathcal{A}$  satisfies the conditions of Schwarz lemma in $\D$. Thus,  $\RE ((1+z)f_1(z)/z) >0$ in $\D$  and hence $f_1 \in \mathcal{F}_b^3$ which shows  the class  $\mathcal{F}_b^3$ is non-empty and $f_1$  is the extremal function which work to prove the sharpness of the radius estimates for this class.

In this manuscript, we determine radius estimates of the functions belonging to the classes $\mathcal{F}_{b,c}^1$, $\mathcal{F}_{b,c}^2$ and  $\mathcal{F}_b^3$ to be in the subclasses $S^*(\a)$, $S_L^* $, $S_p $, $S_e^*$, $S_c^* $, $S_{\sin}^* $, $S_{\leftmoon}^*$, $S_R^*$, $S_{RL}^*$, $S_{\gamma}^*$, $S_{N_e}^*$ and  $S_{SG}^*$ of starlike functions. We also show
some connections of obtained radius estimates with the existing ones.

\section{Generalized Radius Problems}
The class of analytic functions
$p(z) = 1+b_1z+b_2 z^2+ \cdots$ fulfilling $\RE(p(z))>\a$ where $0 \le \a <1$, is denoted by $\mathcal{P}(\a)$. It is noted that $\mathcal{P}(0)=\mathcal{P}$. For $p \in \mathcal{P}(\a) $\cite[P.170]{nehari}, we have
\[\vert b_1\vert \le 2(1-\a).\]
Denote by $P_b(\a)$ the subclass of $\mathcal{P}(\a)$ which contains following type functions
\[p(z) = 1+ 2b(1-\a)z+ \cdots, \vert b\vert \le 1\]
and let $\mathcal{P}_b:= \mathcal{P}_b(0)$.
In order to establish  our main results, following result is required:
\begin{lemma} \cite[Theorem 2, p. 213]{carty}\label{lem1}
Let $\vert b\vert \le 1$ and $0 \le \a <1$. If $p \in \mathcal{P}_b(\a)$, then for $\vert z\vert=r <1$,
\begin{equation*} \label{eq1}
\left\vert \frac{zp'(z)}{p(z)}\right\vert \le \frac{2(1-\a)r}{1-r^2} \frac{|b|r^2+2r+|b|}{(1-2\a)r^2+2(1-\a)|b|r+1}.
\end{equation*}
\end{lemma}

We first determine best possible Ma-Minda type radius estimates for the class $\mathcal{F}_{b,c}^1 $.

\begin{theorem}\label{ThmF1}
Let $d=\vert5b-3c\vert\leq2$ and $c'=\vert3c+1\vert\leq2$. For the class $\mathcal{F}_{b,c}^1 $,  the following radius results hold:
\begin{enumerate}
\item  The $S^*(\a) $- radius is the smallest root $r_1 \in (0,1)$ of the equation
\begin{align}\label{eqa1}
-\a r^6&+(1+(1-\a)(c'+d))r^5+(7-\a+c'+d+(2-\a)c'd)r^4 \notag\\
&+(2+5(c'+d)+c'd)r^3+(6+\a+c'+d+(1+\a)c'd)r^2\notag\\
&+(1+\a(c'+d))r+\a-1=0.
\end{align}
\item The $S_L^* $- radius is the smallest  root $r_2 \in (0,1)$ of the equation
\begin{align}\label{eql1}
\sqrt{2}r^6&+(1+(1+\sqrt{2})(c'+d))r^5+(9+\sqrt{2}+c'+d+\sqrt{2}(\sqrt{2}+1)c'd)r^4\notag\\
&+(2+7(c'+d)+c'd)r^3+(10-\sqrt{2}+c'+d+(3-\sqrt{2})c'd)r^2 \notag\\
&+(1+\sqrt{2}(\sqrt{2}-1)(c'+d))r+1-\sqrt{2}=0.
\end{align}

\item The $S_p $-radius is the smallest  root $r_3 \in (0,1)$ of the equation
\begin{align}\label{eqp1}
-r^6&+(2+c'+d)r^5+(13+2(c'+d)+3c'd)r^4+(4+10(c'+d)+2c'd)r^3 \notag \\
&+(13+2(c'+d)+3c'd)r^2+(2+c'+d)r-1=0.
\end{align}

\item The $S_e^* $- radius is the smallest root $r_4 \in (0,1)$ of the equation
\begin{align} \label{eqe1}
-r^6&+(e+(e-1)(c'+d))r^5+(7e-1+e(c'+d)+(2e-1)c'd)r^4 \notag \\
&+(2e+5e(c'+d)+ec'd)r^3+(6e+1+e(c'+d)+(e+1)c'd)r^2\notag\\
&+(e+c'+d)r+1-e=0.
\end{align}

\item The $S_c^* $- radius is the smallest root $r_5 \in (0,1)$ of the equation
\begin{align} \label{eqc1}
&-r^6+(3+2(c'+d))r^5+(20+3(c'+d)+5c'd)r^4\notag \\
&+(6+15(c'+d)+3c'd)r^3 +(19+3(c'+d)+4c'd)r^2\notag\\
&+(3+c'+d)r-2=0.
\end{align}

\item The $S_{\sin}^* $- radius is the smallest  root $r_6 \in (0,1)$ of the equation
\begin{align}\label{eqs1}
&(1+\sin1)r^6+(1+(2+\sin1)(c'+d))r^5\notag\\
&+(10+\sin1+c'+d+(3+\sin1)c'd)r^4 \notag \\
&+(2+7(c'+d)+c'd)r^3+(9-\sin1+c'+d+(2-\sin1)c'd)r^2 \notag \\
& +(1+(1-\sin1)(c'+d))r-\sin1=0.
\end{align}

\item The $S_{\leftmoon}^*$- radius is the smallest  root $r_7 \in (0,1)$ of the equation
\begin{align}\label{eqm1}
&-(\sqrt{2}-1)r^6+(1+\sqrt{2}(\sqrt{2}-1)(c'+d))r^5\notag\\
&+(8-\sqrt{2}+c'+d+(3-\sqrt{2})c'd)r^4+(2+5(c'+d)+c'd)r^3+(5+\sqrt{2}\notag\\
&+c'+d+\sqrt{2}c'd)r^2  +(1+(\sqrt{2}-1)(c'+d))r-2+\sqrt{2}=0.
\end{align}
\item The $S_R^*$- radius is the smallest  root $r_8 \in (0,1)$ of the equation

\begin{align}\label{eqr1}
&-(2\sqrt{2}-2)r^6+(1+(3-2\sqrt{2})(c'+d))r^5+(9-2\sqrt{2}\notag\\
&+c'+d+(4+2\sqrt{2})c'd)r^4+(2+5(c'+d)+c'd)r^3+(4+2\sqrt{2}+c'\notag\\
&+d+(2\sqrt{2}-1)c'd)r^2  +(1+(2\sqrt{2}-2)(c'+d))r+2\sqrt{2}-3=0.
\end{align}

\item The $S_{RL}^*$- radius is the smallest  root $r_9 \in (0,1)$ of the equation
\begin{align}\label{eqrl1}
&((1+c'+d)r^5+(8+c'+d+2c'd)r^4+(2+6(c'+d)+c'd)r^3\notag \\
&+(8+c'+d+2c'd)r^2+(1+c'+d)r )^2-(r^4+(c'+d)r^3+(2+c'd)r^2\notag \\
&+(c'+d)r+1)^2((1-r^2)\sqrt{(1-r^2)^2-((\sqrt{2}-\sqrt{2}r^2)-1)^2}\notag\\
&-(1-r^2)^2+((\sqrt{2}-\sqrt{2}r^2)-1)^2)=0.
\end{align}

\item The $S_{\gamma}^*$- radius is the smallest  root $r_{10} \in (0,1)$ of the equation
\begin{align*}
(1+c'+d)r ^5&+(8-\sin(\pi \gamma/2)+c'+d+2c'd)r^4 \\
&+(2+(6-\sin(\pi \gamma/2))(c'+d)+c'd)r^3\\
&+(8-2\sin(\pi \gamma/2)+c'+d+(2-\sin(\pi \gamma/2))c'd)r^2\\
&+(1+(1-\sin(\pi \gamma/2))(c'+d))r-\sin(\pi \gamma/2)= 0.
\end{align*}
\item The $S_{N_e}^*$- radius is the smallest  root $r_{11} \in (0,1)$ of the equation
\begin{align}\label{eqn1}
&5r^6+(3+8(c'+d))r^5+(32+3(c'+d)+11c'd)r^4 \nonumber \\
&+(6+21(c'+d)+3c'd)r^3+(25+3(c'+d)+4c'd)r^2 \nonumber\\
&+(3+c'+d)r-2=0.
\end{align}
\item The $S_{SG}^*$- radius is the smallest root $r_{12} \in (0,1)$ of the equation
\begin{align}\label{eqsg1}
&2er^6+(1+e+(1+3e)(c'+d))r^5+(9+11e+(1+e)(c'+d)\notag\\
&+2(1+2e)c'd)r^4
+((1+e)(2+c'd)+7(1+e)(c'+d))r^3\notag \\
&+(2(5+4e)+(1+e)(c'+d)+(3+e)c'd)r^2\notag\\
& +((1+e)+2(c'+d))r+1-e=0.
\end{align}
\end{enumerate}
\end{theorem}

\begin{proof}
Let  $f \in \mathcal{F}_{b,c}^1$ and  $g\in \mathcal{A}$ such that
\begin{equation} \label{eqfg}
\RE \left(\frac{f(z)}{g(z)} \right)>0\,\quad \mbox{and} \quad \RE \left(\frac{1+z}{z}g(z) \right)>0,\,\, (z \in \D).
\end{equation}
Define the functions $p,h: \D \to \mathbb{C}$ by
\begin{equation} \label{eqp(z)}
p(z) = \frac{1+z}{z}g(z) =  1+(1+3c)z+ \cdots
\end{equation}
and
\begin{equation} \label{eqh(z)}
h(z) = \frac{f(z)}{g(z)} =  1+(5b-3c)z+ \cdots .
\end{equation}
By (\ref{eqfg}), (\ref{eqp(z)}) and (\ref{eqh(z)}), it is observed that  the functions $p \in \mathcal{P}_{(3c+1)/2}$ and $h \in \mathcal{P}_{(5b-3c)/2}$.
Thus, we have
\[f(z) = \dfrac{zp(z)h(z)}{1+z}\]
so that
\begin{equation}\label{eqf(z)}
\frac{zf'(z)}{f(z)} =\frac{zp'(z)}{p(z)}+ \frac{zh'(z)}{h(z)}+\frac{1}{1+z}.
\end{equation}
The bilinear mapping $\dfrac{1}{1+z}$ maps  $\vert z \vert \le r$ onto
\begin{equation} \label{eqtr}
\left\vert\frac{1}{1+z}-\frac{1}{1-r^2}\right\vert\le \frac{r}{1-r^2}.
\end{equation}
On taking $\a = 0$ in Lemma \ref{lem1}, we have
\begin{equation} \label{eqt1}
\left\vert\frac{zp'(z)}{p(z)}\right\vert \leq \frac{r}{1-r^2} \frac{c'r^2+4r+c'}{r^2+c'r+1}\,\,
\mbox{and}\,\,
\left\vert\frac{zh'(z)}{h(z)}\right\vert\le \frac{r}{1-r^2} \frac{dr^2+4r+d}{r^2+dr+1}.
\end{equation}
Using (\ref{eqtr}) and (\ref{eqt1}) from (\ref{eqf(z)}), we obtain
\begin{align}\label{eqdisk}
&\left\vert\frac{zf'(z)}{f(z)}- \frac{1}{1-r^2}\right\vert\notag\\ & \le \frac{r}{1-r^2}\left(\frac{c'r^2+4r+c'}{r^2+c'r+1}+\frac{dr^2+4r+d}{r^2+dr+1}+1\right)\notag\\
&= \frac{\splitfrac{(1+c'+d)r^5+(8+c'+d+2c'd)r^4+(2+6(c'+d)+c'd)r^3}{+(8+c'+d+2c'd)r^2+(1+c'+d)r}}{(1-r^2)(r^4+(c'+d)r^3+(2+c'd)r^2+(c'+d)r+1)}
\end{align}
It follows by the above inequality that
\begin{align}
&\RE\left(\frac{zf'(z)}{f(z)}\right)\notag\\ &\ge\frac{\splitfrac{-(1+c'+d)r^5-(7+c'+d+2c'd)r^4-(2+5(c'+d)+c'd)r^3}{-(6+c'+d+c'd)r^2-r+1}}{-r^6-(c'+d)r^5-(1+c'd)r^4+(1+c'd)r^2+(c'+d)r+1} \notag\\ &\ge 0\notag
\end{align}
whenever,
\[(1+c'+d)r^5+(7+c'+d+2c'd)r^4+(2+5(c'+d)+c'd)r^3+(6+c'+d+c'd)r^2+r-1 \le 0.\]

\begin{enumerate}
\item In view of (\ref{eqdisk}), it is noted that
\begin{align}
&\RE\left(\frac{zf'(z)}{f(z)}\right) \notag\\ &\ge\frac{\splitfrac{-(1+c'+d)r^5-(7+c'+d+2c'd)r^4-(2+5(c'+d)+c'd)r^3}{-(6+c'+d+c'd)r^2-r+1}}{-r^6-(c'+d)r^5-(1+c'd)r^4+(1+c'd)r^2+(c'+d)r+1}  \notag\\
&\ge \a \notag
\end{align}
whenever
\begin{align*}
&-\a r^6+(1+(1-\a)(c'+d))r^5+(7-\a+c'+d+(2-\a)c'd)r^4\\
&+(2+5(c'+d)+c'd)r^3+(6+\a+c'+d+(1+\a)c'd)r^2\\
&+(1+\a(c'+d))r+\a-1 \le 0.
\end{align*}
Thus,  the $S^*(\a)$-radius for the class $\mathcal{F}_{b,c}^1$ is the smallest root $r_1 \in (0,1)$ of (\ref{eqa1}).
Consider the functions $f_0$ and $g_0$ which are defined as
\begin{align} \label{ext1}
f_0(z) =& \frac{z(1-z)^2 (1+z)}{(1-(5b-3c)z+z^2)(1-(3c+1)z+z^2)}\qquad \text{and} \notag\\
g_0(z) =& \frac{z(1-z)}{(1-(3c+1)z+z^2)}.
\end{align}
For the function $f_0$ with the choice of $g_0$ given by (\ref{ext1}), we have
\[\frac{f_0(z)}{g_0(z)}= \frac{1-w_1(z)}{1+w_1(z)}\quad  \text{and}\quad \frac{1+z}{z}g_0(z) = \frac{1-w_2(z)}{1+w_2(z)}\]
where
\[ w_1(z) = \dfrac{z(z-(5b-3c)/2)}{(1-(5b-3c)z/2)} \quad \text{with} \quad \vert5b-3c\vert \le 2 \] and
\[w_2(z) = \dfrac{z(z-(3c+1)/2)}{(1-(3c+1)z/2)} \quad \text{with} \quad \vert3c+1\vert \le 2\]
are analytic functions satisfying the conditions of Schwarz lemma in $\D$ and thus $\RE (f_0(z)/g_0(z)) >0$ and $\RE ((1+z)g_0(z)/z) >0$ in $\D$. Hence the function $f_0 \in \mathcal{F}_{b,c}^1$.
At $z=\rho = r_1$,  we have
\begin{align*}
&\frac{zf_0'(z)}{f_0(z)}\\
&=\frac{\splitfrac{5b\rho^5+(-5b(1+6c)+6(-1+c+3c^2))\rho^4}{+(3+20b+3c-15bc+9c^2)\rho^3+(-5+(3-15b)c+9c^2)\rho^2-\rho+1}}{(1-\rho^2)(1-(5b-3c)\rho+\rho^2)(1-(3c+1)\rho+\rho^2)} \\
&= \a,
\end{align*}
shows the sharpness of the result.

\item From (\ref{eqdisk}), it follows that
\begin{align*}
&\left\vert\frac{zf'(z)}{f(z)}- 1\right\vert\notag\\
  &\le  \frac{\splitfrac{r^6+(1+2(c'+d))r^5+(10+c'+d+3c'd)r^4}{+(2+7(c'+d)+c'd)r^3+(9+c'+d+2c'd)r^2+(1+c'+d)r}}{(-r^6-(c'+d)r^5-(1+c'd)r^4+(1+c'd)r^2+(c'+d)r+1)}.
\end{align*}
By using   \cite[Lemma 2.2, p.6559]{ali},
the function $f\in  \mathcal{F}_{b,c}^1$ belongs to the class $S^*_L$ if
\begin{align*}
&\frac{\splitfrac{r^6+(1+2(c'+d))r^5+(10+c'+d+3c'd)r^4+(2+7(c'+d)+c'd)r^3}{+(9+c'+d+2c'd)r^2+(1+c'+d)r}}{(-r^6-(c'+d)r^5-(1+c'd)r^4+(1+c'd)r^2+(c'+d)r+1)} \nonumber\\ & \le \sqrt{2}-1
\end{align*}

 equivalently

\begin{align*}
\sqrt{2}r^6+ &(1+(1+\sqrt{2})(c'+d))r^5+(9+\sqrt{2}+c'+d+\sqrt{2}(\sqrt{2}+1)c'd)r^4\\
&+(2+7(c'+d)+c'd)r^3+(10-\sqrt{2}+c'+d+(3-\sqrt{2})c'd)r^2\\
& +(1+\sqrt{2}(\sqrt{2}-1)(c'+d))r+1-\sqrt{2} \le 0.
\end{align*}
The $S_L^*$-radius for the class $\mathcal{F}_{b,c}^1$ is the smallest  root $r_2 \in (0,1)$ of (\ref{eql1}).
Define the functions $f_0,g_0:\D \to  \mathbb{C} $ as
\begin{align} \label{ext2}
&f_0(z) = \frac{z(1+(5b-3c)z+z^2)(1+(3c+1)z+z^2)}{(1-z^2)^2(1+z)},\notag\\
&g_0(z) = \frac{z(1+ (3c+1)z+z^2)}{(1-z^2)(1+z)}.
\end{align}
For the function $f_0$ with the choice of $g_0$ given by (\ref{ext2}), we have
\[\frac{f_0(z)}{g_0(z)}= \frac{1+w_1(z)}{1-w_1(z)}\,\, \text{and}\,\, \frac{1+z}{z}g_0(z) = \frac{1+w_2(z)}{1-w_2(z)}\]
where $w_1(z) = \dfrac{z(z+(5b-3c)/2)}{(1+(5b-3c)z/2)}$ with $\vert5b-3c\vert \le 2$ and \\ $w_2(z) = \dfrac{z(z+(3c+1)/2)}{(1+(3c+1)z/2)}$ with $\vert3c+1\vert \le 2$ are analytic functions satisfying the conditions of Schwarz lemma in $\D$ and thus $\RE (f_0(z)/g_0(z)) >0$ and $\RE ((1+z)g_0(z)/z) >0$ in $\D$. Hence the function $f_0 \in \mathcal{F}_{b,c}^1$ and the number $\rho = r_2$ satisfies
\begin{align*}
&\frac{\splitfrac{-5b\rho^5+(8-6c+5b(1+6c)-18c^2)\rho^4-(5-15b(-2+c)+3c+9c^2)\rho^3}{+(9(1-c-3c^2)+5b(2+9c))\rho^2-(1+10b)\rho+1}}{(1-\rho^2)(1-(5b-3c)\rho+\rho^2)(1-(3c+1)\rho+\rho^2)}\\
&= \sqrt{2}.\end{align*}
Thus,  for the function $f_0$ defined by (\ref{ext2}) at $z =-\rho = -r_2$,  we have
\begin{align*}
&\left\vert\left(\frac{zf_0'(z)}{f_0(z)}\right)^2-1\right\vert\\
&=\left\vert\left(\frac{\splitfrac{-5b\rho^5+(8-6c-18c^2+5b(1+6c))\rho^4}{\splitfrac{-(5-15b(-2+c)+3c+9c^2)\rho^3}{+(9(1-c-3c^2)+5b(2+9c))\rho^2-(1+10b)\rho+1}}}{(1-\rho^2)(1-(5b-3c)\rho+\rho^2)(1-(3c+1)\rho+\rho^2)}\right)^2-1\right\vert\\
&= 1
\end{align*}
which shows  the  sharpness.

\item Using \cite[Section 3˘, p. 321]{TN},
the disk (\ref{eqdisk}) lies inside the parabolic region $\Omega_p= \{w:\RE w >\vert w-1\vert\}$ provided that
\begin{align*}
&\frac{\splitfrac{(1+c'+d)r^5+(8+c'+d+2c'd)r^4+(2+6(c'+d)+c'd)r^3}{+(8+c'+d+2c'd)r^2+(1+c'+d)r}}{(1-r^2)(r^4+(c'+d)r^3+(2+c'd)r^2+(c'+d)r+1)}\\
&\le \frac{1}{1-r^2}-\frac{1}{2}.
\end{align*}
The above inequality simplifies to
\begin{align*}
&-r^6+(2+c'+d)r^5+(13+2(c'+d)+3c'd)r^4\\ &+(4+10(c'+d)+2c'd)r^3+(13+2(c'+d)+3c'd)r^2+(2+c'+d)r-1\\&\le 0.
\end{align*}
The $S_p$-radius for the class $\mathcal{F}_{b,c}^1$ is the smallest root $r_3 \in (0,1)$ of (\ref{eqp1}).
At $z= \rho = r_3$, we have
\begin{align*}
&\RE \left(\frac{zf_0'(z)}{f_0(z)}\right)\\
&=\frac{\splitfrac{-\rho^6+\rho^5-(-5-15bc+3c+9c^2)\rho^4-(3-15bc+20b+3c+9c^2)\rho^3}{-(-6-30bc-5b+6c+18c^2)\rho^2-5b\rho}}{(1-\rho^2)(1-(5b-3c)\rho+\rho^2)(1-(3c+1)\rho+\rho^2)}\\
&=\left\vert\frac{\splitfrac{\rho^6-\rho^5+(-5-15bc+3c+9c^2)\rho^4+(3-15bc+20b+3c+9c^2)\rho^3}{+(-6-30bc-5b+6c+18c^2)\rho^2+5b\rho}}{(1-\rho^2)(1-(5b-3c)\rho+\rho^2)(1-(3c+1)\rho+\rho^2)}\right\vert\\
&= \left\vert\frac{zf'(z)}{f(z)}-1\right\vert
\end{align*}
which proves the sharpness.

\item By \cite[Lemma 2.2, p.\ 368]{mnr},
the function $f$ in the class $ \mathcal{F}_{b,c}^1$ belongs to the class $S_e^*$ if
\begin{align*}
&\frac{\splitfrac{(1+c'+d)r^5+(8+c'+d+2c'd)r^4+(2+6(c'+d)+c'd)r^3}{+(8+c'+d+2c'd)r^2+(1+c'+d)r}}{(1-r^2)(r^4+(c'+d)r^3+(2+c'd)r^2+(c'+d)r+1)}\\
&\le \frac{1}{1-r^2}-\frac{1}{e}
\end{align*}
equivalently
\begin{align*}
&-r^6+(e+(e-1)(c'+d))r^5+(7e-1+e(c'+d)+(2e-1)c'd)r^4  \\
&+(2e+5e(c'+d)+ec'd)r^3+(6e+1+e(c'+d)+(e+1)c'd)r^2\\
&+(e+c'+d)r+1-e \le 0.
\end{align*}
Thus,  the $S_e^*$-radius for the class $\mathcal{F}_{b,c}^1$ is the smallest positive root $r_4 \in (0,1)$ of (\ref{eqe1}).
The estimate is sharp for the function $f_0$ defined in (\ref{ext1}). At $z=\rho=r_4$, we have
\begin{align*}
\left\vert\log \frac{zf_0'(z)}{f_0(z)}\right\vert  &= \left\vert\log \left(\frac{\splitfrac{5b\rho^5+(-5b(1+6c)+6(-1+c+3c^2))\rho^4}{\splitfrac{+(3+20b+3c-15bc+9c^2)\rho^3}{+(-5+(3-15b)c+9c^2)\rho^2-\rho+1}}}
{(1-\rho^2)(1-(5b-3c)\rho+\rho^2)(1-(3c+1)\rho+\rho^2)}\right)\right\vert\\
&=1.
\end{align*}

\item
By making use of \cite[Lemma 2.5, p.926]{sharma},
the function $f$ in the class $ \mathcal{F}_{b,c}^1$ belongs to the class $ S_c^*$ if
\begin{align*}
&\frac{\splitfrac{(1+c'+d)r^5+(8+c'+d+2c'd)r^4+(2+6(c'+d)+c'd)r^3}{+(8+c'+d+2c'd)r^2+(1+c'+d)r}}{(1-r^2)(r^4+(c'+d)r^3+(2+c'd)r^2+(c'+d)r+1)}\\
&\le \frac{1}{1-r^2}-\frac{1}{3}
\end{align*}
 equivalently
\begin{align*}
& -r^6+(3+2(c'+d))r^5+(20+3(c'+d)+5c'd)r^4  \\
&+(6+15(c'+d)+3c'd)r^3+(19+3(c'+d)+4c'd)r^2+(3+c'+d)r-2\\& \le 0.
\end{align*}
The $S_c^*$-radius for the class $\mathcal{F}_{b,c}^1$ is the smallest root $r_5 \in (0,1)$ of (\ref{eqc1}).
The result is sharp for the function $f_0$ defined in (\ref{ext1}). At $z= \rho=r_5$, we have
\begin{align*}
\left\vert \frac{zf_0'(z)}{f_0(z)}\right\vert  &= \left\vert \frac{\splitfrac{5b\rho^5+(-5b(1+6c)+6(-1+c+3c^2))\rho^4}{\splitfrac{+(3+20b+3c-15bc+9c^2)\rho^3}{+(-5+(3-15b)c+9c^2)\rho^2-\rho+1}}}
{(1-\rho^2)(1-(5b-3c)\rho+\rho^2)(1-(3c+1)\rho+\rho^2)} \right\vert\\
&=\frac{1}{3}= \phi_c(-1).
\end{align*}\

\item
Using \cite[Lemma 3.3, p.219]{cho},
the function $f$ in the class $ \mathcal{F}_{b,c}^1$ belongs to the class $  S_{\sin}^*$ if
\begin{align*}
&\frac{\splitfrac{(1+c'+d)r^5+(8+c'+d+2c'd)r^4+(2+6(c'+d)+c'd)r^3}{+(8+c'+d+2c'd)r^2+(1+c'+d)r}}{(1-r^2)(r^4+(c'+d)r^3+(2+c'd)r^2+(c'+d)r+1)}\\
&\le \sin 1- \frac{r^2}{1-r^2}
\end{align*}
or equivalently if,
\begin{align*}
&(1+\sin1)r^6 +(1+(2+\sin1)(c'+d))r^5+(10+\sin1+c'+d\\ & +(3+\sin1)c'd)r^4  +(2+7(c'+d)+c'd)r^3+(9-\sin1+c'+d\\
&+(2-\sin1)c'd)r^2  +(1+(1-\sin1)(c'+d))r-\sin1 \le 0.
\end{align*}
Thus,  the $S_{\sin}^*$-radius for the class $\mathcal{F}_{b,c}^1$ is the smallest root $r_6 \in (0,1)$ of (\ref{eqs1}).
At $z=-\rho=-r_6$, we have
\begin{align*}
\left\vert\frac{zf_0'(z)}{f_0(z)}\right\vert &= \left\vert\frac{\splitfrac{-5b\rho^5+(8-6c-18c^2+5b(1+6c))\rho^4}{\splitfrac{+(-5+15b(-2+c)-3c-9c^2)\rho^3}{+(9(1-c-3c^2)+5b(2+9c))\rho^2-(1+10b)\rho+1}}}{(1-\rho^2)(1-(5b-3c)\rho+\rho^2)(1-(3c+1)\rho+\rho^2)}\right\vert\\
& = 1+\sin1
\end{align*}
which proves sharpness.

\item By \cite[Lemma 2.1, p. 3]{gandhi1},
the function $f$ in the class $ \mathcal{F}_{b,c}^1$ belongs to the class $  S_{\leftmoon}^*$ if
\begin{align*}
&\frac{\splitfrac{(1+c'+d)r^5+(8+c'+d+2c'd)r^4+(2+6(c'+d)+c'd)r^3}{+(8+c'+d+2c'd)r^2+(1+c'+d)r}}{(1-r^2)(r^4+(c'+d)r^3+(2+c'd)r^2+(c'+d)r+1)}\\
&\le 1-\sqrt{2}+\frac{1}{1-r^2}
\end{align*}
which simplifies to
\begin{align*}
&-(\sqrt{2}-1)r^6+ (1+\sqrt{2}(\sqrt{2}-1)(c'+d))r^5+(8-\sqrt{2}+c'+d\\&+(3-\sqrt{2})c'd)r^4
+(2+5(c'+d)+c'd)r^3+(5+\sqrt{2}+c'+d+\sqrt{2}c'd)r^2 \notag \\
& +(1+(\sqrt{2}-1)(c'+d))r-2+\sqrt{2} \le 0.
\end{align*}
Thus,  the $S_{\leftmoon}^*$-radius for the class $\mathcal{F}_{b,c}^1$ is the smallest root $r_7 \in (0,1)$ of (\ref{eqm1}).
The bound is best possible for the function $f_0$ defined by (\ref{ext1}). At $z = \rho = r_7$, we have
\begin{align*}
&\left\vert\left( \frac{zf_0'(z)}{f_0(z)}\right)^2-1\right\vert \\& = \left\vert\left(\frac{\splitfrac{5b\rho^5+(-5b(1+6c)+6(-1+c+3c^2))\rho^4}{\splitfrac{+(3+20b+3c-15bc+9c^2)\rho^3}{+(-5+(3-15b)c+9c^2)\rho^2-\rho+1}}}{(1-\rho^2)(1-(5b-3c)\rho+\rho^2)(1-(3c+1)\rho+\rho^2)}\right)^2-1\right\vert\\
 &=\vert(\sqrt{2}-1)^2-1\vert= 2\vert\sqrt{2}-1\vert=2\left\vert\frac{zf_0'(z)}{f_0(z)}\right\vert .
\end{align*}

\item By \cite[Lemma 2.2 p. 202]{kumar},
the function $f$ in the class $ \mathcal{F}_{b,c}^1$ belongs to the class $ S_R^*$ if
\begin{align*}
&\frac{\splitfrac{(1+c'+d)r^5+(8+c'+d+2c'd)r^4+(2+6(c'+d)+c'd)r^3}{+(8+c'+d+2c'd)r^2+(1+c'+d)r}}{(1-r^2)(r^4+(c'+d)r^3+(2+c'd)r^2+(c'+d)r+1)}\\
&\le \frac{1}{1-r^2}+ 2-2 \sqrt{2}
\end{align*}
which simplifies to
\begin{align*}
&-(2\sqrt{2}-2)r^6+ (1+(3-2\sqrt{2})(c'+d))r^5+(9-2\sqrt{2}+c'+d\\
&+(4-2\sqrt{2})c'd)r^4
+(2+5(c'+d)+c'd)r^3+(4+2\sqrt{2}+c'+d\\
&+(2\sqrt{2}-1)c'd)r^2
 +(1+(2\sqrt{2}-2)(c'+d))r+2\sqrt{2}-3 \le 0.
\end{align*}
The $S_R^*$-radius for the class $\mathcal{F}_{b,c}^1$ is the smallest root $r_8 \in (0,1)$ of (\ref{eqr1}).
The obtained estimate is best possible for the function $f_0$ defined by (\ref{ext1}). As for $z=\rho = r_8$, we have
\begin{align*}
&\left\vert \frac{zf_0'(z)}{f_0(z)}\right\vert \\& = \left\vert\frac{\splitfrac{5b\rho^5+(-5b(1+6c)+6(-1+c+3c^2))\rho^4}{\splitfrac{+(3+20b+3c-15bc+9c^2)\rho^3+(-5+(3-15b)c+9c^2)\rho^2}{-\rho+1}}}{(1-\rho^2)(1-(5b-3c)\rho+\rho^2)(1-(3c+1)\rho+\rho^2)}\right\vert\\
 &=2(\sqrt{2}-1) = \psi(-1).
\end{align*}

\item By \cite[Lemma 3.2, p. 10]{mnr1},
the function $f$ in the class $ \mathcal{F}_{b,c}^1$ belongs to the class $  S_{RL}^*$ if
\begin{align*}
&\frac{\splitfrac{(1+c'+d)r^5+(8+c'+d+2c'd)r^4+(2+6(c'+d)+c'd)r^3}{+(8+c'+d+2c'd)r^2+(1+c'+d)r}}{(1-r^2)(r^4+(c'+d)r^3+(2+c'd)r^2+(c'+d)r+1)}\\
&\qquad \qquad \le ((1-(\sqrt{2}-\frac{1}{1-r^2})^2)^{1/2}-(1-(\sqrt{2}-\frac{1}{1-r^2})^2))^{1/2}
\end{align*}
which simplifies to
\begin{align*}
&((1+c'+d)r^5+(8+c'+d+2c'd)r^4+(2+6(c'+d)+c'd)r^3 \\&+(8+c'+d+2c'd)r^2+(1+c'+d)r )^2-(r^4+(c'+d)r^3+(2+c'd)r^2\\
&+(c'+d)r+1)^2((1-r^2)\sqrt{(1-r^2)^2-((\sqrt{2}-\sqrt{2}r^2)-1)^2} \\
&-(1-r^2)^2+((\sqrt{2}-\sqrt{2}r^2)-1)^2) \le 0.
\end{align*}
Thus,  the $S_{RL}^*$-radius for the class $\mathcal{F}_{b,c}^1$ is the smallest root $r_9 \in (0,1)$ of (\ref{eqrl1}).

\item By \cite[Lemma 3.1, p. 307]{grs},
the function $f$ in the class $ \mathcal{F}_{b,c}^1$ belongs to the class $  S_{\gamma}^*$ if
\begin{align*}
&\frac{\splitfrac{(1+c'+d)r^5+(8+c'+d+2c'd)r^4+(2+6(c'+d)+c'd)r^3}{+(8+c'+d+2c'd)r^2+(1+c'+d)r}}{(1-r^2)(r^4+(c'+d)r^3+(2+c'd)r^2+(c'+d)r+1)}\\
&\le \frac{1}{1-r^2}\sin(\pi \gamma/2)
\end{align*}
which yields
\begin{align*}
&(1+c'+d)r^5+(8-\sin(\pi \gamma/2)+c'+d+2c'd)r^4\\
&+(2+(6-\sin(\pi \gamma/2))(c'+d)+c'd)r^3  +(8-2\sin(\pi \gamma/2)+c'+d\\
&+(2-\sin(\pi \gamma/2))c'd)r^2 +(1+(1-\sin(\pi \gamma/2))(c'+d))r-\sin(\pi \gamma/2)\le 0.
\end{align*}

\item Using \cite[Lemma 2.2, p. 86]{wani},
the function $f$ in the class $ \mathcal{F}_{b,c}^1$ belongs to the class $  S^*_{N_e}$ if
\begin{align*}
&\frac{\splitfrac{(1+c'+d)r^5+(8+c'+d+2c'd)r^4+(2+6(c'+d)+c'd)r^3}{+(8+c'+d+2c'd)r^2+(1+c'+d)r}}{(1-r^2)(r^4+(c'+d)r^3+(2+c'd)r^2+(c'+d)r+1)}\\
&\le \frac{5}{3}- \frac{1}{1-r^2}
\end{align*}
equivalently
\begin{align*}
&5r^6+(3+8(c'+d))r^5+(32+3(c'+d)+11c'd)r^4\\
&+(6+21(c'+d)+3c'd)r^3
+(25+3(c'+d)+4c'd)r^2+(3+c'+d)r-2\\
& \le 0.
\end{align*}
The $S^*_{N_e}$-radius for the class $\mathcal{F}_{b,c}^1$ is the smallest root $r_{11} \in (0,1)$ of (\ref{eqn1}).
At $z=-\rho=-r_{11}$, we have
\begin{align*}
\left\vert\frac{zf_0'(z)}{f_0(z)}\right\vert &= \left\vert\frac{\splitfrac{-5b\rho^5+(8-6c-18c^2+5b(1+6c))\rho^4}{\splitfrac{+(-5+15b(-2+c)-3c-9c^2)\rho^3}{+(9(1-c-3c^2)+5b(2+9c))\rho^2-(1+10b)\rho+1}}}{(1-\rho^2)(1-(5b-3c)\rho+\rho^2)(1-(3c+1)\rho+\rho^2)}\right\vert\\
& = \frac{5}{3} \,\,\text{which lies in the boundary of}\,\, \Omega_{N_e}
\end{align*}
which  shows the sharpness.

\item Using \cite[Lemma 2.2, p. 961]{goel},
the function $f$ in the class $ \mathcal{F}_{b,c}^1$ belongs to the class $ S^*_{SG}$ if
\begin{align*}
&\frac{\splitfrac{(1+c'+d)r^5+(8+c'+d+2c'd)r^4+(2+6(c'+d)+c'd)r^3}{+(8+c'+d+2c'd)r^2+(1+c'+d)r}}{(1-r^2)(r^4+(c'+d)r^3+(2+c'd)r^2+(c'+d)r+1)}\\
&\le \frac{2e}{1+e}- \frac{1}{1-r^2}
\end{align*}
which gives
\begin{align*}
2er^6&+(1+e+(1+3e)(c'+d))r^5+(9+11e+(1+e)(c'+d)\notag\\
&+2(1+2e)c'd)r^4+((1+e)(2+c'd)+7(1+e)(c'+d))r^3+(2(5+4e)\notag\\
&+(1+e)(c'+d)+(3+e)c'd)r^2 +((1+e)+2(c'+d))r+1-e \le 0.
\end{align*}
The $S^*_{SG}$-radius for the class $\mathcal{F}_{b,c}^1$ is the smallest root $r_{12} \in (0,1)$ of (\ref{eqsg1}).
For sharpness consider the function $f_0$ defined in (\ref{ext2}). Let $w=zf_0'(z)/f_0(z)$. At $z=-\rho=-r_{12}$,we have
\begin{align*}
&\left\vert\log\left(\frac{w}{2-w}\right)\right\vert\\
&=\left\vert\log\left(\frac{\splitfrac{5b\rho^5+(-8+6c+18c^2-5b(1+6c))\rho^4}{\splitfrac{-(-5+15b(-2+c)-3c-9c^2)\rho^3}{+(-5b(2+9c)+9(-1+c+3c^2))\rho^2
-(1+10b)\rho-1}}}{\splitfrac{2\rho^6-(2+15b)\rho^5+(10+15b-12c+60bc-36c^2)\rho^4}{\splitfrac{-(5-15b(-2+c)+3c+9c^2)\rho^3}{+(7+3(-1+5b)c-9c^2)\rho^2
+\rho-1}}}\right)\right\vert\\
& = 1.
\end{align*}
This shows that the radius is sharp.\qedhere
\end{enumerate}
\end{proof}

\begin{remark}
On taking $b=-1$ and $c=-1$, parts (1)-(10) of Theorem \ref{ThmF1} reduces to  \cite[Theorem 2.1]{AV} .
\end{remark}

Next theorem yields the radius estimates of various starlikeness for the class $\mathcal{F}_{b,c}^2 $.

\begin{theorem}\label{ThmF2}
Let $d'=\vert3c-4b\vert\leq2$ and $c'=\vert3c+1\vert\leq2$. For the class $\mathcal{F}_{b,c}^2 $, the following radius results hold:
\begin{enumerate}
\item  The $S^*(\a)$-radius is the smallest  root $r_1 \in (0,1)$ of the equation
\begin{align}\label{eqa2}
(1-\a)d'r^5&+((2-\a)(1+c'd')+d')r^4+(1+(3-\a)c'+5d'+c'd')r^3 \notag \\
&+(5+c'+d'+(1+\a)c'd')r^2+(1+\a(c'+d'))r+\a-1=0.
\end{align}
\item The $S_L^*$-radius is the smallest  root $r_2 \in (0,1)$ of the equation
\begin{align}\label{eql2}
&(1+\sqrt{2})r^5+(\sqrt{2}(1+\sqrt{2})(1+c'd')+d')r^4\notag\\
&+(1+(3+\sqrt{2})c'+7d'+c'd'))r^3 +(7+c'+d'+(3-\sqrt{2})c'd')r^2\notag\\
&+(1+\sqrt{2}(\sqrt{2}-1)(c'+d'))r+1-\sqrt{2}=0.
\end{align}

\item The $S_p$- radius is the smallest  root $r_3 \in (0,1)$ of the equation
\begin{align*}
d'r^5&+(3+2d'+3c'd')r^4+(2+5c'+10d'+2c'd')r^3 \notag \\
&+(10+2(c'+d')+3c'd')r^2+(2+c'+d')r-1=0.
\end{align*}

\item The $S_e^*$- radius is the smallest root $r_4 \in (0,1)$ of the equation
\begin{align*}
&(e-1)d'r^5 +(2e-1+ed'+(2e-1)c'd')r^4 \notag \\
&+(e+(3e-1)c'+5ed'+ec'd')r^3 +(5e+e(c'+d')+(e+1)c'd')r^2\notag \\
&+(e+c'+d')r+1-e=0.
\end{align*}

\item The $S_c^*$- radius is the smallest  root $r_5 \in (0,1)$ of the equation
\begin{align} \label{eqc2}
2d'r^5&+(5+3d'+5c'd')r^4 +(3+8c'+15d'+3c'd')r^3\notag \\
&+(15+3(c'+d')+4c'd')r^2+(3+c'+d')r-2=0.
\end{align}

\item The $S_{\sin}^*$- radius is the smallest  root $r_6 \in (0,1)$ of the equation
\begin{align*}
&(2+\sin 1)d'r^5+((3+\sin 1)(1+c'd')+d')r^4\notag \\
 &+(1+(4+\sin 1)c'+7d'+c'd')r^3+(7+c'+d'+(2-\sin1)c'd')r^2\notag \\
& +(1+(1-\sin 1)(c'+d'))r-\sin 1=0.
\end{align*}

\item The $S_{\leftmoon}^*$- radius is the smallest root $r_7 \in (0,1)$ of  the equation
\begin{align*}
 &\sqrt{2}(\sqrt{2}-1)d'r^5+((3-\sqrt{2})(1+c'd')+d')r^4 \notag \\
 & +(1+(4-\sqrt{2})c'+5d'+c'd')r^3 +(5+c'+d'+\sqrt{2}c'd')r^2\notag \\
 &+(1+(\sqrt{2}-1)(c'+d'))r+\sqrt{2}(1-\sqrt{2})=0.
\end{align*}
\item The $S_R^*-$ radius is the smallest root $r_8 \in (0,1)$ of the equation
\begin{align}\label{eqr2}
 &(3-2\sqrt{2})d'r^5+((4-2\sqrt{2})(1+c'd')+d')r^4  \notag \\
 & +(1+(5-2\sqrt{2})c'+5d'+c'd')r^3+(5+c'+d'+(2\sqrt{2}-1)c'd')r^2\notag \\ &+(1+(2\sqrt{2}-2)(c'+d'))r+2\sqrt{2}-3=0.
\end{align}

\item The $S_{RL}^*$- radius is the smallest root $r_9 \in (0,1)$ of  the equation
\begin{align}\label{eqrl2}
&(d'r^5+(2+d'+2c'd')r^4+(1+3(c'+2d')+c'd')r^3\notag \\& +(6+c'+d'+2c'd')r^2+(1+c'+d')r )^2
-(d'r^3+(1+c'd')r^2\notag \\&+(c'+d')r+1)^2((1-r^2)\sqrt{(1-r^2)^2-((\sqrt{2}-\sqrt{2}r^2)-1)^2}\notag \\ &-(1-r^2)^2+((\sqrt{2}-\sqrt{2}r^2)-1)^2)=0.
\end{align}
\item The $S_{\gamma}^*$-radius is the smallest  root $r_{10} \in (0,1)$ of  the equation
\begin{align*}
d' r ^5&+(2+d'+2c'd')r ^4+(1+3c'+(6-\sin(\pi \gamma/2))d'+c'd') r ^3 \notag\\
& +(6-\sin(\pi \gamma/2)+c'+d'+(2-\sin(\pi \gamma/2))c'd') r ^2 \notag\\
&+(1+(1-\sin(\pi \gamma/2))(c'+d'))r-\sin(\pi \gamma/2)=0.
\end{align*}
\item The $S_{N_e}^*$-radius is the smallest  root $r_{11} \in (0,1)$ of the equation
\begin{align}\label{eqn2}
8d'r^5&+(11(1+c'd')+3d')r^4+(3+7(2c'+3d')+3c'd')r^3 \notag \\
&+(21+3(c'+d')+4c'd')r^2+(3+c'+d')r-2=0.
\end{align}
\item The $S_{SG}^*$-radius is the smallest root $r_{12} \in (0,1)$ of the equation
\begin{align*}
&(1+3e)d'r^5+(2(1+2e)(1+c'd')+(1+e)d')r^4+((1+e)(1+c'd')\notag \\
&+(3+5e)c' +7(1+e)d')r^3+(7(1+e)+(1+e)(c'+d')+(3+e)c'd')r^2\notag\\
&+(1+e+2(c'+d'))r+1-e=0.
\end{align*}
\end{enumerate}
\end{theorem}

\begin{proof}
Let  $f\in \mathcal{F}_{b,c}^2$ and  $g\in \mathcal{A}$ be two functions such that
\begin{equation} \label{eqfg1}
\quad\left\vert\frac{f(z)}{g(z)}-1\right\vert<1\quad \mbox{and}\quad \RE \left(\frac{1+z}{z}g(z) \right)>0\,\, (z \in \D).
\end{equation}
The functions $p,h: \D \to \mathbb{C}$ by
\begin{equation} \label{eqp(z)1}
p(z) = (1+\frac{1}{z})g(z) =  1+(1+3c)z+ \cdots
\end{equation}
and
\begin{equation} \label{eqh(z)1}
h(z) = \frac{g(z)}{f(z)} =  1+(3c-4b)z+ \cdots .
\end{equation}
By (\ref{eqfg1}), (\ref{eqp(z)1}) and (\ref{eqh(z)1}), it is noted that the functions $p \in \mathcal{P}_{(3c+1)/2}$ and $h \in \mathcal{P}_{(3c-4b)}(1/2)$. Also
\[f(z) = \dfrac{g(z)}{h(z)}=\dfrac{zp(z)}{(1+z)h(z)}.\]
On differentiating  logarithmically, we get
\begin{equation}\label{eqf(z)1}
\frac{zf'(z)}{f(z)} =\frac{zp'(z)}{p(z)}- \frac{zh'(z)}{h(z)}+\frac{1}{1+z}.
\end{equation}
On taking $\a = 0$ and $\a=1/2$ in Lemma \ref{lem1}, we have
\begin{equation} \label{eqt2}
\left\vert\frac{zp'(z)}{p(z)}\right\vert \le \frac{r}{1-r^2} \frac{c'r^2+4r+c'}{r^2+c'r+1}\,\,
\mbox{and}\,\,
\left\vert\frac{zh'(z)}{h(z)}\right\vert\le \frac{r}{1-r^2} \frac{d'r^2+2r+d'}{d'r+1}.
\end{equation}
Using (\ref{eqtr}) and (\ref{eqt2}) from (\ref{eqf(z)1}), we obtain
\begin{align}\label{eqdisk2}
\left\vert\frac{zf'(z)}{f(z)}- \frac{1}{1-r^2}\right\vert & \le \frac{r}{1-r^2}\left(\frac{c'r^2+4r+c'}{r^2+c'r+1}+\frac{d'r^2+2r+d'}{d'r+1}+1\right)\notag\\
&= \frac{\splitfrac{d'r^5+(2+d'+2c'd')r^4+(1+3(c'+2d')+c'd')r^3}{+(6+c'+d'+2c'd')r^2+(1+c'+d')r}}{(1-r^2)(d'r^3+(1+c'd')r^2+(c'+d')r+1)}.
\end{align}
It follows by the above inequality that
\[\RE\left(\frac{zf'(z)}{f(z)}\right) \ge\frac{\splitfrac{-d'r^5-(2+d'+2c'd')r^4-(1+3c'+5d'+c'd')r^3}{-(5+c'+d'+c'd')r^2-r+1}}{-d'r^5-(1+c'd')r^4-c'r^3+c'd'r^2+(c'+d')r+1}  \ge 0\]
whenever,
\[d'r^5+(2+d'+2c'd')r^4+(1+3c'+5d'+c'd')r^3+(5+c'+d'+c'd')r^2+r-1 \le 0.\]

\begin{enumerate}
\item In view of (\ref{eqdisk}), we have
\begin{align*}
\RE\left(\frac{zf'(z)}{f(z)}\right) &\ge\frac{\splitfrac{1-d'r^5-(2+d'+2c'd')r^4-(1+3c'+5d'+c'd')r^3}{-(5+c'+d'+c'd')r^2-r}}{-d'r^5-(1+c'd')r^4-c'r^3+c'd'r^2+(c'+d')r+1}\\
&  \ge \a
\end{align*}
whenever
\begin{align*}
(1-\a)d'r^5&+((2-\a)(1+c'd')+d')r^4+(1+(3-\a)c'+5d'+c'd')r^3\\
&+(5+c'+d'+(1+\a)c'd')r^2+(1+\a(c'+d'))r+\a-1 \le 0.
\end{align*}
The $S^*(\a)$-radius for the class $\mathcal{F}_{b,c}^2$ is the smallest root $r_1 \in (0,1)$ of (\ref{eqa2}).

\item Using (\ref{eqdisk2}),  we have
\begin{align*}
\left\vert\frac{zf'(z)}{f(z)}- 1\right\vert &\le \left\vert\frac{zf'(z)}{f(z)}- \frac{1}{1-r^2}\right\vert+\frac{r^2}{1-r^2} \notag\\
 &\le  \frac{\splitfrac{2d'r^5+(3(1+c'd')+d')r^4+(1+4c'+7d'+c'd')r^3}{+(7+c'+d'+2c'd')r^2+(1+c'+d')r}}{-d'r^5-(1+c'd')r^4-c'r^3+c'd'r^2+(c'+d')r+1}.
\end{align*}
By \cite[Lemma 2.2, p. 6559]{ali}, the function $f \in \mathcal{F}_{b,c}^2$ belongs to the class $S^*_L$, whenever the following inequality holds
\[\frac{\splitfrac{2d'r^5+(3(1+c'd')+d')r^4+(1+4c'+7d'+c'd')r^3}{+(7+c'+d'+2c'd')r^2+(1+c'+d')r}}{-d'r^5-(1+c'd')r^4-c'r^3+c'd'r^2+(c'+d')r+1} \le \sqrt{2}-1\]
 equivalently
\begin{align*}
 &(1+\sqrt{2})d'r^5+(\sqrt{2}(\sqrt{2}+1)(1+c'd')+d')r^4\\
&+(1+(3+\sqrt{2})c'+7d'+c'd')r^3+(7+c'+d'+(3-\sqrt{2})c'd')r^2\\
& +(1+\sqrt{2}(\sqrt{2}-1)(c'+d'))r+1-\sqrt{2} \le 0.
\end{align*}
The $S_L^*-$ radius for the class $\mathcal{F}_{b,c}^2$ is the smallest root $r_2 \in (0,1)$ of (\ref{eql2}).

\item Using \cite[Section 3˘, p. 321]{TN},
the disk (\ref{eqdisk2}) lies in  the parabolic domain $\Omega_p= \{w:\RE w >\vert{w-1\vert}\}$ provided that
\[\frac{\splitfrac{d'r^5+(2+d'+2c'd')r^4+(1+3(c'+2d')+c'd')r^3}{+(6+c'+d'+2c'd')r^2+(1+c'+d')r}}{(1-r^2)(d'r^3+(1+c'd')r^2+(c'+d')r+1)}
\le \frac{1}{1-r^2}-\frac{1}{2}\]
which on simplification becomes
\begin{align*}
d'r^5&+(3+2d'+3c'd')r^4+(2+5(c'+2d')+2c'd')r^3\\ &+(10+2(c'+d)+3c'd')r^2+(2+c'+d')r-1\le 0.
\end{align*}
Thus, we get the desired $S_p$- radius for the class $\mathcal{F}_{b,c}^2$.

\item By \cite[Lemma 2.2, p.\ 368]{mnr},
the function $f \in \mathcal{F}_{b,c}^2$  to be in the class $ S_e^*$ if
\[\frac{\splitfrac{d'r^5+(2+d'+2c'd')r^4+(1+3(c'+2d')+c'd')r^3}{+(6+c'+d'+2c'd')r^2+(1+c'+d')r}}{(1-r^2)(d'r^3+(1+c'd')r^2+(c'+d')r+1)}\le \frac{1}{1-r^2}-\frac{1}{e}\]
or
\begin{align*}
&(e-1)d'r^5+(2e-1+ed'+(2e-1)c'd')r^4 \\
& +(e+(3e-1)c'+5ed'+ec'd')r^3 +(5e+e(c'+d')+(e+1)c'd')r^2\\
&+(e+c'+d')r+1-e \le 0.
\end{align*}
Thus, we get the desired $S_e^*$-radius.

\item By making use of \cite[Lemma 2.5, p.926]{sharma},
the function $f \in \mathcal{F}_{b,c}^2$  to be in the class $ S_c^*$ if
\[\frac{\splitfrac{d'r^5+(2+d'+2c'd')r^4+(1+3(c'+2d')+c'd')r^3}{+(6+c'+d'+2c'd')r^2+(1+c'+d')r}}{(1-r^2)(d'r^3+(1+c'd')r^2+(c'+d')r+1)}\le \frac{1}{1-r^2}-\frac{1}{3}\]
or equivalently, if
\begin{align*}
& 2d'r^5+(5+3d'+5c'd')r^4 +(3+8c'+15d'+3c'd')r^3 \\
&+(15+3(c'+d')+4c'd')r^2+(3+c'+d')r-2 \le 0.
\end{align*}
The $S_c^* $- radius is the smallest e root $r_5 \in (0,1)$ of (\ref{eqc2}).

\item Using \cite[Lemma 3.3, p.219]{cho},
the function $f \in \mathcal{F}_{b,c}^2$  to be in the class $ S_{\sin}^*$ if
\[\frac{\splitfrac{d'r^5+(2+d'+2c'd')r^4+(1+3(c'+2d')+c'd')r^3}{+(6+c'+d'+2c'd')r^2+(1+c'+d')r}}{(1-r^2)(d'r^3+(1+c'd')r^2+(c'+d')r+1)}\le \sin 1- \frac{r^2}{1-r^2}\]
or equivalently if,
\begin{align*}
&(2+\sin1)d'r^5+((3+\sin1)(1+c'd')+d')r^4 \\
& +(1+(4+\sin 1)c'+7d'+c'd')r^3+(7+c'+d'+(2-\sin1)c'd')r^2\\
&  +(1+(1-\sin1)(c'+d'))r-\sin1 \le 0.
\end{align*}
Thus,  we get the desired  $S_{\sin}^*$- radius.

\item By \cite[Lemma 2.1, p. 3]{gandhi1},
the function $f \in \mathcal{F}_{b,c}^2$  to be in the class $ S_{\leftmoon}^*$ if
\[\frac{\splitfrac{d'r^5+(2+d'+2c'd')r^4+(1+3(c'+2d')+c'd')r^3}{+(6+c'+d'+2c'd')r^2+(1+c'+d')r}}{(1-r^2)(d'r^3+(1+c'd')r^2+(c'+d')r+1)}\le 1-\sqrt{2}+\frac{1}{1-r^2}\]
which on simplification becomes
\begin{align*}
 &\sqrt{2}(\sqrt{2}-1)d'r^5+((3-\sqrt{2})(1+c'd')+d')r^4\\
&+(1+(4-\sqrt{2})c'+5d'+c'd')r^3+(5+c'+d'+\sqrt{2}c'd')r^2\\
&+(1+(\sqrt{2}-1)(c'+d'))r+\sqrt{2}(1-\sqrt{2}) \le 0.
\end{align*}
Thus,  we get the desired   $S_{\leftmoon}^* $-radius.

\item By \cite[Lemma 2.2 p. 202]{kumar},
the function $f \in \mathcal{F}_{b,c}^2$ belongs to the class $S_R^*$ if
\[\frac{\splitfrac{d'r^5+(2+d'+2c'd')r^4+(1+3(c'+2d')+c'd')r^3}{+(6+c'+d'+2c'd')r^2+(1+c'+d')r}}{(1-r^2)(d'r^3+(1+c'd')r^2+(c'+d')r+1)}\le \frac{1}{1-r^2}+ 2-2 \sqrt{2}\]
equivalently
\begin{align*}
 &(3-2\sqrt{2})d'r^5+((4-2\sqrt{2})(1+c'd')+d')r^4\\
& +(1+(5-2\sqrt{2})c'+5d'+c'd')r^3+(5+c'+d'+(2\sqrt{2}-1)c'd')r^2\\
& +(1+(2\sqrt{2}-2)(c'+d'))r+
2\sqrt{2}-3 \le 0.
\end{align*}
The $S_R^*$ -radius is the smallest root $r_8 \in (0,1)$ of (\ref{eqr2}).

\item By using \cite[Lemma 3.2, p. 10]{mnr1},
the function $f \in \mathcal{F}_{b,c}^2$ belongs to the class $S_{RL}^*$ if
\begin{align*}
&\frac{\splitfrac{d'r^5+(2+d'+2c'd')r^4+(1+3(c'+2d')+c'd')r^3}{+(6+c'+d'+2c'd')r^2+(1+c'+d')r}}{(1-r^2)(d'r^3+(1+c'd')r^2+(c'+d')r+1)}\\
&\le ((1-(\sqrt{2}-\frac{1}{1-r^2})^2)^{1/2}-(1-(\sqrt{2}-\frac{1}{1-r^2})^2))^{1/2}
\end{align*}
which simplifies to
\begin{align*}
&(d'r^5+(2+d'+2c'd')r^4+(1+3(c'+2d')+c'd')r^3\\
& +(6+c'+d'+2c'd')r^2+(1+c'+d')r )^2
-(d'r^3+(1+c'd')r^2\\
& +(c'+d')r+1)^2((1-r^2)\sqrt{(1-r^2)^2-((\sqrt{2}-\sqrt{2}r^2)-1)^2}-(1-r^2)^2\\
&+((\sqrt{2}-\sqrt{2}r^2)-1)^2) \le 0.
\end{align*}
The $S_{RL}^*$-radius  is the smallest  root $r_9 \in (0,1)$ of (\ref{eqrl2}).

\item By \cite[Lemma 3.1, p. 307]{grs},
the function $f \in \mathcal{F}_{b,c}^2$ belongs to the class $  S_{\gamma}^*$ if
\[\frac{\splitfrac{d'r^5+(2+d'+2c'd')r^4+(1+3(c'+2d')+c'd')r^3}{+(6+c'+d'+2c'd')r^2+(1+c'+d')r}}{(1-r^2)(d'r^3+(1+c'd')r^2+(c'+d')r+1)}\le \frac{1}{1-r^2}\sin(\pi \gamma/2)\]
which gives
\begin{align*}
d'r^5& +(2+d'+2c'd')r^4+(1+3c'+(6-\sin(\pi \gamma/2))d'+c'd')r^3 \\
& +(6-\sin(\pi \gamma/2)+c'+d'+(2-\sin(\pi \gamma/2))c'd')r^2 \\
&+(1+(1-\sin(\pi \gamma/2))(c'+d'))r-\sin(\pi \gamma/2)\le 0.
\end{align*}

\item Using \cite[Lemma 2.2, p. 86]{wani},
the function $f \in \mathcal{F}_{b,c}^2$  to be in the class $  S^*_{N_e}$ if
\[\frac{\splitfrac{d'r^5+(2+d'+2c'd')r^4+(1+3(c'+2d')+c'd')r^3}{+(6+c'+d'+2c'd')r^2+(1+c'+d')r}}{(1-r^2)(d'r^3+(1+c'd')r^2+(c'+d')r+1)}\le \frac{5}{3}-\frac{1}{1-r^2}\]
equivalently
\begin{align*}
8d'r^5&+(11(1+c'd')+3d')r^4+(3+7(2c'+3d')+3c'd')r^3 \notag \\
&+(21+3(c'+d')+4c'd')r^2+(3+c'+d')r-2 \le 0.
\end{align*}
Thus,  the $S^*_{N_e} $- radius for the class $\mathcal{F}_{b,c}^2$ is the smallest positive root $r_{11} \in (0,1)$ of (\ref{eqn2}).

\item Using \cite[Lemma 2.2, p. 961]{goel},
the function $f \in \mathcal{F}_{b,c}^2$  to be in the class $  S^*_{SG}$ if the following inequality holds
\[\frac{\splitfrac{d'r^5+(2+d'+2c'd')r^4+(1+3(c'+2d')+c'd')r^3}{+(6+c'+d'+2c'd')r^2+(1+c'+d')r}}{(1-r^2)(d'r^3+(1+c'd')r^2+(c'+d')r+1)}\le \frac{2e}{1+e}-\frac{1}{1-r^2}\]
equivalently
\begin{align*}
&(1+3e)d'r^5+(2(1+2e)(1+c'd')+(1+e)d')r^4
+((1+e)(1+c'd')\\
&+(3+5e)c'+7(1+e)d')r^3 +(7(1+e)+(1+e)(c'+d')\\
&+(3+e)c'd')r^2+(1+e+2(c'+d'))r+1-e\le 0.
\end{align*}
Thus, we get the desired  $S^*_{SG}$ -radius. \qedhere
\end{enumerate}
\end{proof}

\begin{remark}
For $b=-1$ and $c=-1$, parts (1)-(10) of Theorem \ref{ThmF2} becomes \cite[Theorem 2.4]{AV} .
\end{remark}

The next result provides certain starlikeness for the class $\mathcal{F}_b^3 $ and proof is obtained as the proof of Theorem \ref{ThmF2}, so it is  not given.
\begin{theorem}\label{ThmF3}
Let $b'=\vert 1+3b\vert\leq2$. For the class $\mathcal{F}_b^3 $,  the following radius results hold:
\begin{enumerate}
\item  The $S^*(\a)$- radius is the smallest  root $r_1 \in (0,1)$ of the equation
\begin{equation*}\label{eqa3}
-\a r^4+(1+(1-\a)b')r^3  +(3+b')r^2+(1+\a b')r+\a-1=0.
\end{equation*}
\item The $S_L^* $- radius is the smallest  root $r_2 \in (0,1)$ of the equation
\begin{equation*}\label{eql3}
\sqrt{2}r^4+(1+(1+\sqrt{2})b')r^3+(5+b')r^2 +(1+\sqrt{2}(\sqrt{2}-1)b')r+1-\sqrt{2}=0.
\end{equation*}
The result is sharp.
\item The $S_p $- radius is the smallest  root $r_3 \in (0,1)$ of the equation
\begin{equation*}\label{eqp3}
-r^4+(2+b')r^3+2(3+b')r^2+(2+b')r-1=0.
\end{equation*}

\item The $S_e^*$- radius is the smallest  root $r_4 \in (0,1)$ of the equation
\begin{equation*} \label{eqe3}
-r^4+(e+(e-1)b')r^3 +e(3+b')r^2+(e+b')r+1-e=0.
\end{equation*}
\end{enumerate}
\end{theorem}

\begin{remark}
On taking $b=-1$  in parts (1)-(4) of Theorem \ref{ThmF3}, we obtain \cite[Theorem 2.6]{AV} .
\end{remark}



\end{document}